\documentclass[reqno]{amsart}
\usepackage{amsmath}
\usepackage{amssymb, latexsym, mathrsfs,  mathabx, dsfont, a4wide}
\usepackage{hyperref}

\makeatletter
\@namedef{subjclassname@2020}{\textup{2020} Mathematics Subject Classification}
\makeatother

 \newtheorem{theorem}{Theorem}[section]
 
 \newtheorem{proposition}[theorem]{Proposition}

 \newtheorem*{theorem*}{Theorem}
\newtheorem*{proposition*}{Proposition}
\newtheorem*{lemma*}{Lemma}

\theoremstyle{definition}
 \newtheorem{definition}[theorem]{Definition}

 \theoremstyle{remark}
 
 \newtheorem{remark}[theorem]{Remark}

   \newtheorem*{claim*}{Claim}

\newcommand{\op}[1]{\operatorname{#1}}

\newcommand{\scal}[2]{\ensuremath{\left\langle #1 | #2 \right\rangle}}

\def\XXint#1#2#3{{\setbox0=\hbox{$#1{#2#3}{\int}$}
\vcenter{\hbox{$#2#3$}}\kern-.5\wd0}}

\newcommand{\C}{\ensuremath{\mathbb{C}}}

\newcommand{\R}{\ensuremath{\mathbb{R}}} 
\newcommand{\T}{\ensuremath{\mathbb{T}}} 
\newcommand{\Z}{\ensuremath{\mathbb{Z}}}

\newcommand{\Rn}{\ensuremath{\R^{n}}}

\newcommand{\cH}{\ensuremath{\mathscr{H}}}

\newcommand{\cL}{\ensuremath{\mathscr{L}}}

\newcommand{\cS}{\ensuremath{\mathscr{S}}}

\newcommand{\scL}{\mathscr{L}}

\newcommand{\Sp}{\op{Sp}}

\begin{document}
\title[Logarithmic Sobolev Inequalities on Noncommutative Tori]
{Logarithmic Sobolev Inequalities of Fractional Order on Noncommutative Tori}

\author[Gihyun Lee\hfil \hfilneg] {Gihyun Lee}  

\address{Gihyun Lee \newline
Department of Mathematics: Analysis, Logic and Discrete Mathematics, Ghent University, Krijgslaan 281, Building S8, B 9000 Ghent, Belgium}
\email{Gihyun.Lee@ghent.ac.kr}

\begin{abstract}
In this paper, we prove a version of the logarithmic Sobolev inequality of fractional order on noncommutative $n$-tori for any dimension $n\geq 2$. \end{abstract}

\maketitle \numberwithin{equation}{section}
\allowdisplaybreaks

\section{Introduction}

Among all noncommutative spaces studied in Alain Connes' noncommutative geometry program~\cite{Co:NCG} noncommutative tori are the most extensively studied ones. One of the reasons behind the extensive research conducted on noncommutative tori is the presence of counterparts to various mathematical tools employed in the study of analysis and geometry. For example, counterparts to the notions such as vector bundles, Fourier series, spaces of smooth functions, Sobolev spaces and pseudodifferential calculus are established and available in the setting of noncommutative tori (see, e.g., \cite{CXY:CMP13, Co:CRAS80, HLP:Part1, HLP:Part2, PS:CMP03, XXY:MAMS18} and the references therein). Moreover, noncommutative tori are utilized in the mathematical modeling of physical phenomena, such as the quantum Hall effect~\cite{BES:JMP94}, topological insulators~\cite{BCR:RMP16, PS:Springer16} and string theory~\cite{CDS:JHEP98, SW:JHEP99}.

Let us now briefly review the classical Sobolev inequalities on Euclidean spaces. The classical Sobolev inequality states that if a function $f$ defined on $\Rn$, along with its first derivatives, belongs to $L_p(\Rn)$ and if $q = (\frac{1}{p}-\frac{1}{n})^{-1}$ is finite, then $f$ is in $L_q(\Rn)$. This inequality has broad applications in various fields of analysis, including the study of partial differential equations. However, the classical Sobolev inequality has the following limitations. As evident from the value of $q$ we see that, as the dimension $n$ increases, the difference between $q$ and $p$ narrows. Consequently, as $n$ grows, the improvement in summability obtained from the differentiability decreases. For this reason, in the infinite dimensional setting such as quantum field theory we cannot expect to obtain an inequality precisely corresponds to the classical Sobolev inequality. Motivated by this observation Gross~\cite{Gr:AJM75} introduced and proved the following logarithmic Sobolev inequality on Euclidean spaces.
\begin{equation*}
\int_{\Rn} |f(x)|^2 \log{|f(x)|} \, d\nu(x) \leq \int_{\Rn} |\nabla f(x)|^2 \, d\nu(x) + \|f\|_2^2\log{\|f\|_2} .
\end{equation*}
Here $\nu$ denotes the Gaussian measure on $\Rn$ and $\|\cdot\|_2$ denotes the $L_2$-norm associated with $\nu$. As mentioned in~\cite{Gr:AJM75} this inequality can be utilized in the infinite dimensional setting, because the coefficients in the inequality do not depend on the dimension $n$.

Since the introduction of Gross' logarithmic Sobolev inequality, various methods have been employed to establish logarithmic Sobolev inequalities in different settings. Given the vast number of papers on this topic, it is not feasible to encompass all the references here. However, for a partial overview, let us list a few results which can be found in the literature. Rosen presented and proved a logarithmic Sobolev inequality on weighted $\Rn$~\cite{Ro:TAMS76}. Weissler proved a logarithmic Sobolev inequality on the circle~\cite{We:JFA80}. Gross~\cite{Gr:IJM92} and Chatzakou-Kassymov-Ruzhansky~\cite{CKR:arXiv21} investigated logarithmic Sobolev inequalities on Lie groups. Stroock-Zegarlinski~\cite{SZ:JFA92} and Bodineau-Helffer~\cite{BH:JFA99} obtained logarithmic Sobolev inequalities for spin systems. Brannan-Gao-Junge derived a version of logarithmic Sobolev inequality by utilizing lower bound of the Ricci curvature of a compact Riemannian manifold~\cite{BGJ:AM22}.

In the case of noncommutative tori, as mentioned above, there exist counterparts to the tools used in Fourier theory for ordinary tori (see~\cite{CXY:CMP13}). Therefore, the arguments employed on ordinary tori can often be adapted to the setting of noncommutative tori. By using this harmonic analysis technique of noncommutative tori and the theory of operator algebras Xiong-Xu-Yin provided a detailed account on the construction of Sobolev, Besov and Triebel-Lizorkin spaces on noncommutative tori~(\cite{XXY:MAMS18}; see also~\cite{HLP:Part2, Ro:APDE08, Sp:Padova92} for Sobolev spaces on noncommutative tori). The embedding theorems between these spaces are also proved in~\cite{XXY:MAMS18}. In addition, in~\cite{MP:JMP22, MP:AM23} McDonald-Ponge proved versions of Sobolev inequalities on noncommutative tori as consequences of the Lieb-Thirring inequalities on (curved) noncommutative tori.

However, to the best of the author's knowledge, logarithmic Sobolev inequalities in the setting of noncommutative tori haven't been studied much in the literature. The logarithmic Sobolev inequality on noncommutative $2$-tori by Khalkhali-Sadeghi~\cite{KS:JPDOA17} is the only existing result on this topic. They attempted to establish logarithmic Sobolev inequality on noncommutative $2$-tori by adopting Weissler's proof of the logarithmic Sobolev inequality on the circle~\cite{We:JFA80}. However, due to technical issues arising from the noncommutativity, they were only able to obtain the following form of logarithmic Sobolev inequality for strictly positive elements of the form $x = \sum_{k\in\Z} x_k U_1^kU_2^{kl}$, where $0\neq l\in\Z$.
\begin{equation*}
\tau\big( x^2 \log{x} \big) \leq \sum_{k\in\Z} (1+|l|) \, |k| \, |x_k|^2 + \|x\|_{L_2}^2 \log{\|x\|_{L_2}} .
\end{equation*}
We refer to Section~\ref{sec:Preliminaries} for notations and background material on noncommutative tori.

In this paper, we prove the following logarithmic Sobolev inequalities of fractional order on noncommutative tori for any dimension $n\geq 2$. 

\begin{theorem} \label{thm:intro.log-Sobolev-ineq}
Let $0<x\in C^\infty(\T_\theta^n)$, $a>0$ and $0<s<\frac{n}{2}$. Then there is a constant $C(n,s,a)>0$ depends only on $n,s$ and $a$ such that
\begin{equation} \label{eq:intro.log-Sobolev-ineq}
\tau\left[ x^2 \log\left( \frac{x^2}{\|x\|_{L_2}^2} \right) \right] \leq C(n,s,a) \, \|x\|_{W_2^s}^2 - \frac{n}{s} (\log{a}+1) \, \|x\|_{L_2}^2 .
\end{equation}
\end{theorem}

Our proof of this theorem is based on the short proof of fractional order logarithmic Sobolev inequality on $\Rn$ presented in a recent article by Chatzakou-Ruzhansky~\cite{CR:arXiv23}. In the setting of noncommutative tori, there are tools correspond to the key tools used in the proof by Chatzakou-Ruzhansky~\cite{CR:arXiv23}, which enables us to apply their approach immediately to noncommutative tori. Instead of Jensen's inequality for concave functions and probability measures used in the proof by Chatzakou-Ruzhansky, we utilize the Jensen's operator inequality known in the operator algebraic setting~(\cite{Ch:IJM74, Da:PAMS57}; see also~\cite{PFMS:Zagreb05}) and the embedding theorem between Sobolev spaces proved in~\cite{XXY:MAMS18} to prove Theorem~\ref{thm:intro.log-Sobolev-ineq}.

Although the main result of this paper, Theorem~\ref{thm:intro.log-Sobolev-ineq}, is the first result on logarithmic Sobolev inequality of \emph{fractional order} on noncommutative tori, it still has certain limitations and room for improvement. First, the Sobolev norm used on the right-hand side of~(\ref{eq:intro.log-Sobolev-ineq}) needs to be replaced by the homogeneous Sobolev norm,
\begin{equation*}
\|x\|_{\dot{W}_2^s(\T_\theta^n)}^2 := \sum_{k\in\Z^n} |k|^{2s} \, |x_k|^2 , \qquad s>0 .
\end{equation*}
However, an inequality between homogeneous Sobolev norms on noncommutative tori which can be applied to the arguments in this paper is missing in the literature. Although McDonald-Ponge's Sobolev inequalities~\cite{MP:JMP22, MP:AM23} deal with homogeneous Sobolev norms, these results cannot be directly applied to the arguments of this paper. In~\cite{MP:JMP22, MP:AM23} Sobolev inequalities are proven only for zero mean value elements, i.e., elements $x$ with $\tau(x) = 0$. But in order for the logarithm defined by holomorphic functional calculus appearing in~(\ref{eq:intro.log-Sobolev-ineq}) to make sense, our focus should be restricted to strictly positive elements, and strictly positive elements cannot have a zero mean. Another issue is the lack of a result on the sharpness of Sobolev inequalities on noncommutative tori. In~\cite{CR:arXiv23} Chatzakou-Ruzhansky utilized the sharp constant of the Sobolev inequality on $\Rn$ obtained in~\cite{CT:JMAA04} to get an explicit expression of the constant appearing in their logarithmic Sobolev inequality and study its behavior. Similarly, if we could determine the sharpness of Sobolev inequalities on noncommutative tori, we could expect to improve the constant on the right-hand side of~(\ref{eq:intro.log-Sobolev-ineq}) or get an explicit expression of it. In particular, the sharpness of Sobolev inequalities would enable us to verify the conjecture on the logarithmic Sobolev inequality on noncommutative $2$-tori stated in~\cite{KS:JPDOA17} by setting $a = \frac{1}{e}$ in~(\ref{eq:intro.log-Sobolev-ineq}).

This paper is organized as follows. In Section~\ref{sec:Preliminaries}, We gather some background material on noncommutative tori, Sobolev spaces and Jensen's operator inequality used in this paper. In Section~\ref{sec:Proof}, we prove the logarithmic Sobolev inequalities of fractional order on noncommutative tori, the main result of this paper (Theorem~\ref{thm:intro.log-Sobolev-ineq}).

\section{Preliminaries} \label{sec:Preliminaries}

In this section, we gather background material on noncommutative tori, Sobolev spaces and Jensen's operator inequality used in the proof of the main theorem in Section~\ref{sec:Proof}.

\subsection{Noncommutative tori}
We refer to~\cite{Co:NCG, HLP:Part1, Ri:CM90} for a more detailed account on noncommutative tori.

Let $\theta$ be a skew-symmetric $n\times n$ matrix over $\R$. The noncommutative torus associated with $\theta$, denoted by $\T_\theta^n$, is a noncommutative space in the sense of Alain Connes' noncommutative geometry~\cite{Co:NCG}. The $C^*$-algebra $C(\T_\theta^n)$ and the von Neumann algebra $L_\infty(\T_\theta^n)$ of $\T_\theta^n$ are generated by the unitaries $U_1,\ldots,U_n$ subject to the relations,
\begin{equation*}
U_kU_j = e^{2\pi i\theta_{jk}}U_jU_k , \qquad 1\leq j,k\leq n .
\end{equation*}
These unitaries $U_1,\ldots,U_n$ can be concretely realized as unitary operators on $L_2(\T^n)$ (see, e.g., \cite{HLP:Part1}), and hence both $C(\T_\theta^n)$ and $L_\infty(\T_\theta^n)$ are $*$-subalgebras of $\cL(L_2(\T^n))$, the $C^*$-algebra of bounded linear operators on $L_2(\T^n)$. If $\theta = 0$, then we recover the spaces of continuous linear functions $C(\T^n)$ and essentially bounded functions $L_\infty(\T^n)$ on the ordinary $n$-torus $\T^n = \Rn/(2\pi\Z)^n$.

In what follows, for any $k = (k_1,\ldots,k_n)\in\Z^n$, we shall denote $U_1^{k_1}\cdots U_n^{k_n}$ by $U^k$. Given any $T\in\cL(L_2(\T^n))$, we define
\begin{equation*}
\tau(T) = {\scal {T1} {1}}_{L_2(\T^n)} = (2\pi)^{-n}\int_{\T^n} (T1)(x) \, dx .
\end{equation*}
This defines a tracial state on both $C(\T_\theta^n)$ and $L_\infty(\T_\theta^n)$. By a direct computation it can be shown that $\tau(U^k) = 0$ for $0\neq k\in\Z^n$ and $\tau(1) = 1$. For $x,y\in C(\T_\theta^n)$ we define
\begin{equation*}
{\scal x y}_{L_2(\T_\theta^n)} = \tau(xy^*) .
\end{equation*}
Let us denote by $L_2(\T_\theta^n)$ the Hilbert space completion of $C(\T_\theta^n)$ with respect to this inner product. Then the family $\{U^k ; \, k\in\Z^n\}$ forms an orthonormal basis for $L_2(\T_\theta^n)$. This orthonormal basis plays the role of the standard orthonormal basis for $L_2(\T^n)$ utilized in Fourier analysis, i.e., every $x\in L_2(\T_\theta^n)$ can be uniquely written as follows.
\begin{equation} \label{eq:preliminaries.Fourier-series}
x = \sum_{k\in\Z^n} x_kU^k , \qquad x_k := {\scal x {U^k}}_{L_2(\T_\theta^n)} .
\end{equation}
Furthermore, the GNS construction (see, e.g., \cite{Ar:Springer81}) associated with $\tau$ gives rise to the $*$-representations of $C(\T_\theta^n)$ and $L_\infty(\T_\theta^n)$ on the Hilbert space $L_2(\T_\theta^n)$.

The $C^*$-algebra $C(\T_\theta^n)$ admits a strongly continuous action of $\R^n$, denoted by $\alpha_s(x)$ for $s\in\Rn$ and $x\in C(\T_\theta^n)$. Hence the triple $(C(\T_\theta^n),\Rn,\alpha)$ forms a $C^*$-dynamical system. For the elements $U^k$, $k\in\Z^n$, this action is given by
\begin{equation*}
\alpha_s(U^k) = e^{is\cdot k}U^k , \qquad s\in\Rn .
\end{equation*}
The $C^*$-dynamical system structure on $C(\T_\theta^n)$ enables us to define the dense subalgebra $C^\infty(\T_\theta^n)$ of smooth elements of the action $\alpha$, i.e., we define
\begin{equation*}
C^\infty(\T_\theta^n) := \left\{ x\in C(\T_\theta^n) ; \, \text{$\Rn\ni s\mapsto\alpha_s(x)\in C(\T_\theta^n)$ is a $C^\infty$-map} \right\} .
\end{equation*}
We also have the following characterization of the smooth elements in $C(\T_\theta^n)$ in terms of the Fourier series expansion given in~(\ref{eq:preliminaries.Fourier-series}).
\begin{equation*}
C^\infty(\T_\theta^n) = \Big\{ x = \sum_{k\in\Z^n} x_kU^k ; \, (x_k)_{k\in\Z^n}\in\cS(\Z^n) \Big\} .
\end{equation*}
Here $\cS(\Z^n)$ denotes the space of rapidly decaying sequences indexed by $\Z^n$ with entries in $\C$. Furthermore, the action $\Rn\ni s\mapsto\alpha_s(x)\in C^\infty(\T_\theta^n)$ also enables us to define the derivations $\partial_1,\ldots,\partial_n$ on $C^\infty(\T_\theta^n)$. For $j=1,\ldots,n$, we define $\partial_j:C^\infty(\T_\theta^n)\rightarrow C^\infty(\T_\theta^n)$ by letting
\begin{equation*}
\partial_j(x) = \partial_{s_j}\alpha_s(x) \big|_{s=0} , \qquad x\in C^\infty(\T_\theta^n) .
\end{equation*}
In particular, for the elements $U^k$, $k = (k_1,\ldots,k_n)\in\Z^n$, we obtain $\partial_j(U^k) = ik_j(U^k)$, $1\leq j\leq n$.

\subsection{Sobolev and $L_p$-spaces}
A detailed account on Sobolev spaces on noncommutative tori can be found in~\cite{HLP:Part2, Ro:APDE08, Sp:Padova92, XXY:MAMS18}.

Let us denote the $L_2$-version of Sobolev space of order $s\geq 0$ on $\T_\theta^n$ by $W_2^s(\T_\theta^n)$. This space consists of elements $x = \sum_{k\in\Z^n} x_kU^k$ in $L_2(\T_\theta^n)$ such that
\begin{equation*}
\big( (1+|k|^2)^{\frac{s}{2}} x_k \big)_{k\in\Z^n}\in\ell_2(\Z^n) .
\end{equation*}
The space $W_2^s(\T_\theta^n)$ is a Hilbert space with the inner product,
\begin{equation*}
{\scal x y}_{W_2^s(\T_\theta^n)} := \sum_{k\in\Z^n} (1+|k|^2)^s x_k \overline{y_k} , \qquad x = \sum_{k\in\Z^n}x_kU^k ,  y = \sum_{k\in\Z^n}y_kU^k \in W_2^s(\T_\theta^n) .
\end{equation*}
We shall denote the norm associated with this inner product by $\|\cdot\|_{W_2^s(\T_\theta^n)}$.

We apply the theory of noncommutative $L_p$-spaces (see, e.g., \cite{Ne:JFA74, Te:Copenhagen81}) to the von Neumann algebra $L_\infty(\T_\theta^n)$ to construct the $L_p$-spaces of $\T_\theta^n$. Recall that the elements of $L_\infty(\T_\theta^n)$ are represented as bounded linear operators on the Hilbert space $L_2(\T^n)$. We say that a closed and densely defined operator on $L_2(\T^n)$ is affiliated with $L_\infty(\T_\theta^n)$ if it commutes with the commutant of $L_\infty(\T_\theta^n)$ in $\cL(L_2(\T^n))$. Given a positive operator $x$ on $L_2(\T^n)$ affiliated with $L_\infty(\T_\theta^n)$ let $x = \int_0^\infty \lambda \, dE(\lambda)$ be its spectral representation and set
\begin{equation*}
\tau(x) := \int_0^\infty \lambda \, d\tau(E(\lambda)) .
\end{equation*}
For $1\leq p<\infty$, the $L_p$-space of $\T_\theta^n$, denoted by $L_p(\T_\theta^n)$, consists of all elements $x$ in the $*$-algebra of $L_\infty(\T_\theta^n)$-affiliated operators on $L_2(\T^n)$ such that
\begin{equation*}
\|x\|_{L_p(\T_\theta^n)} := \tau\big(|x|^p\big)^{\frac{1}{p}}<\infty .
\end{equation*}
The space $L_p(\T_\theta^n)$ is a Banach space with the norm $\|\cdot\|_{L_p(\T_\theta^n)}$.

Furthermore, we have the following inclusion of Sobolev spaces into $L_p$-spaces. This is a particular case of the embedding theorem for Sobolev spaces on noncommutative tori stated in~\cite{XXY:MAMS18}.
\begin{proposition}[{see~\cite[Theorem~6.6]{XXY:MAMS18}}] \label{prop:preliminaries.Sobolev-embedding}
Let $p>2$ and set $s = n(\frac{1}{2}-\frac{1}{p})$. Then there is a continuous embedding,
\begin{equation*}
W_2^s(\T_\theta^n)\subset L_p(\T_\theta^n) .
\end{equation*}
\end{proposition}

%

%
%

%
%
%
%
%
%

\subsection{Jensen's operator inequality}

\begin{definition}
Let $f(t)$ be a function on an interval $I\subset\R$. We say that $f(t)$ is \emph{operator convex} if, for all $\lambda\in[0,1]$ and for every selfadjoint operator $A$ and $B$ on a Hilbert space $\cH$ such that $\Sp{A}\subset I$ and $\Sp{B}\subset I$, we have
\begin{equation*}
f ( (1-\lambda)A + \lambda B ) \leq (1-\lambda)f(A) + \lambda f(B) .
\end{equation*}
In contrast, $f(t)$ is called \emph{operator concave} if $-f(t)$ is operator convex.
\end{definition}

\begin{proposition}[{\cite{Ch:IJM74, Da:PAMS57}; see also~\cite[Theorem~1.20]{PFMS:Zagreb05}}] \label{prop:preliminaries.operator-Jensen}
Let $\cH$ and $\cH'$ be Hilbert spaces. Suppose that $f(t)$ is an operator convex function on the interval $I$, $x\in\scL(\cH)$ and $\Sp{x}\subset I$. Then we have
\begin{equation} \label{eq:preliminaries.operator-Jensen}
\pi(f(x)) \geq f(\pi(x)) ,
\end{equation}
for any positive normalized linear map $\pi:\scL(\cH)\rightarrow\scL(\cH')$.
\end{proposition}

\begin{remark} \label{rem:preliminaries.concave}
The inequality~(\ref{eq:preliminaries.operator-Jensen}) is of course reversed if we employ an operator concave function instead of an operator convex function.
\end{remark}

\section{The Proof of Theorem~\ref{thm:intro.log-Sobolev-ineq}} \label{sec:Proof}

In this section, we provide the proof of Theorem~\ref{thm:intro.log-Sobolev-ineq}, the main result of this paper.

\begin{proof}[Proof of Theorem~\ref{thm:intro.log-Sobolev-ineq}]
Let $0<x\in C^\infty(\T_\theta^n)$. Then, for any $\varepsilon>0$, we have
\begin{equation} \label{eq:proof.introducing-epsilon}
\tau\left[ x^2 \log\left( \frac{x^2}{\|x\|_{L_2}^2} \right) \right] = \frac{1}{\varepsilon} \, \tau\left[ x^2 \log\left( \frac{x^2}{\|x\|_{L_2}^2} \right)^\varepsilon \right] = \frac{\|x\|_{L_2}^2}{\varepsilon} \, \tau\left[ \frac{x^2}{\|x\|_{L_2}^2} \log\left( \frac{x^{2\varepsilon}}{\|x\|_{L_2}^{2\varepsilon}} \right) \right] .
\end{equation}
Note that the map $\scL(L_2(\T^n))\ni u\mapsto\tau\left[ \frac{|x|^2}{\|x\|_{L^2}^2}u \right]\in\C$ is a positive normalized linear map. Moreover, we know by~\cite[Example~1.7]{PFMS:Zagreb05} that the logarithmic function $\log{t}$ is operator concave on $(0,\infty)$. Therefore, as we have $\Sp\big(x^{2\varepsilon}/\|x\|_{L_2}^{2\varepsilon}\big)\subset(0,\infty)$, it follows from Jensen's operator inequality (Proposition~\ref{prop:preliminaries.operator-Jensen} and Remark~\ref{rem:preliminaries.concave}) that we have
\begin{align}
\nonumber \tau\left[ \frac{x^2}{\|x\|_{L_2}^2} \log\left( \frac{x^{2\varepsilon}}{\|x\|_{L_2}^{2\varepsilon}} \right) \right] &\leq \log \tau\left( \frac{x^{2\varepsilon+2}}{\|x\|_{L_2}^{2\varepsilon+2}} \right) \\\nonumber
&= (\varepsilon+1)\log \left\| \frac{x}{\|x\|_{L_2}} \right\|_{L_{2\varepsilon+2}}^2 \\
&= (\varepsilon+1)\log\left( \frac{\|x\|_{L_{2\varepsilon+2}}^2}{\|x\|_{L_2}^2} \right) . \label{eq:proof.applying-Jensen}
\end{align}
For all $0<b,t\in\R$, we have
\begin{equation*}
\log{t} \leq bt - \log{b} - 1 .
\end{equation*}
Combining this with the estimates~(\ref{eq:proof.introducing-epsilon}) and~(\ref{eq:proof.applying-Jensen}) we obtain
\begin{align}
\nonumber \tau\left[ x^2 \log\left( \frac{x^2}{\|x\|_{L_2}^2} \right) \right] &\leq \frac{\|x\|_{L_2}^2(\varepsilon+1)}{\varepsilon} \, \log\left( \frac{\|x\|_{L_{2\varepsilon+2}}^2}{\|x\|_{L_2}^2} \right) \\\nonumber
&\leq \frac{\|x\|_{L_2}^2(\varepsilon+1)}{\varepsilon} \left( b\frac{\|x\|_{L_{2\varepsilon+2}}^2}{\|x\|_{L_2}^2} - \log{b} - 1 \right) \\
&= \frac{\varepsilon+1}{\varepsilon} \left( b \, \|x\|_{L_{2\varepsilon+2}}^2 - [\log{b}+1] \, \|x\|_{L_2}^2 \right) . \label{eq:proof.b-introduced-in-the-estimate}
\end{align}
Let $\varepsilon>0$ be such that $2\varepsilon+2 = \frac{2n}{n-2s}$ for $0<s<\frac{n}{2}$ and set $b = ea^2$ for $a>0$. We know by Proposition~\ref{prop:preliminaries.Sobolev-embedding} that there is $C(n,s)>0$ depending only on $n$ and $s$ such that
\begin{equation*}
\|x\|_{L_{2\varepsilon+2}}^2 \leq C(n,s) \, \|x\|_{W_2^s}^2 .
\end{equation*}
It then follows from this and~(\ref{eq:proof.b-introduced-in-the-estimate}) that
\begin{align*}
\tau\left[ x^2 \log\left( \frac{x^2}{\|x\|_{L_2}^2} \right) \right] &\leq \frac{\frac{n}{n-2s}}{\frac{2s}{n-2s}} \left( ea^2 \, \|x\|_{L_{2\varepsilon+2}}^2 - [\log(ea^2)+1] \, \|x\|_{L_2}^2 \right) \\
&\leq \frac{n}{2s} \left( ea^2 C(n,s) \, \|x\|_{W_2^s}^2 - 2(\log{a}+1) \, \|x\|_{L_2}^2 \right) \\
&= \frac{nea^2}{2s} C(n,s) \, \|x\|_{W_2^s}^2 - \frac{n}{s} (\log{a}+1) \, \|x\|_{L_2}^2 .
\end{align*}
This completes the proof.
\end{proof}

\section*{Acknowledgements}
The author wishes to thank Michael Ruzhansky for helpful discussions related to the topic of this paper. The research for this article is financially supported by the FWO Odysseus 1 grant G.0H94.18N: Analysis and Partial Differential Equations and by the Methusalem programme of the Ghent University Special Research Fund (BOF) (Grant number 01M01021).

\end{document}